\documentclass{amsart}

\usepackage[a4paper, dvips]{geometry}
\usepackage{latexsym}
\usepackage{amsfonts}
\usepackage{amsmath, amssymb}
\usepackage{bbm}
\usepackage{color}
\usepackage{tikz}
\usepackage{cite}
\usepackage{enumerate}

\usepackage{graphicx}

\newtheorem{theorem}{Theorem}

\newtheorem{lemma}{Lemma}
\theoremstyle{remark}
\newtheorem{remark}{Remark}
\theoremstyle{definition}

\newcommand{\NN}{\mathbb{N}}
\newcommand{\FF}{\mathbb{F}}

\begin{document}

\title{Counting appearances of integers in sets of arithmetic progressions}

\author{Florian Pausinger}
\address{Department of Mathematics and CEMSUL, Faculty of Sciences, University of Lisbon, Portugal.}
\email{fpausinger at fc.ul.pt}
\date{}


\begin{abstract}
The sequence $A067549$ of The On-Line Encyclopedia of Integer Sequences \cite{oeis}, is defined as $(a_k)_{k \geq  1}$ with $a_k$ being the determinant of the $k \times k$ matrix whose diagonal contains the first $k$ prime numbers and all other elements are ones.
We relate this sequence to a concrete counting problem.
Choose an arbitrary residue class $r_i$ for each prime $p_i$ with $1 \leq i \leq k$ and set $P_k = \prod_{i=1}^k p_i$. We show that $a_k$ is the number of integers in $[1, P_k]$ that are contained in \emph{at most} one of the $k$ chosen residue classes.
Interestingly, we show that this sequence is closely related to the better known sequence $A005867$ for which we derive a novel characterisation in terms of determinants and which gives the number of integers in $[1, P_k]$ that are not contained in any of the $k$ residue classes.

Our proof is purely structural and, therefore, it can be generalised to counting appearances of integers in residue classes of arbitrary arithmetic progressions generated by $k$ different primes using the determinant of a matrix of ones having those $k$ primes on its diagonal.
The revealed structure also offers a fast way of calculating such determinants.
\\[5pt]
\end{abstract}
\maketitle

\section{Introduction}
The sequence $A067549$ of The On-Line Encyclopedia of Integer Sequences \cite{oeis}, is defined as $(a_k)_{k \geq  1}$ with $a_k$ being the determinant of the $k \times k$ matrix whose diagonal contains the first $k$ prime numbers and all other elements are ones, i.e., 
$$ a_k = \det
 \begin{bmatrix}
2 & 1 & \ldots  & 1 \\
1 & 3 & \ldots & 1 \\
\vdots &  \vdots & \ddots & \vdots \\
1 & 1 & \ldots & p_k 
\end{bmatrix}.  $$
The first five elements, for $k=1, \ldots, 5$, of this sequence are
$$2, 5, 22, 140, 1448, \ldots,$$
see \cite{oeis} for a table of the first 300 elements of this sequence.
We relate this sequence to a concrete counting problem.
Let $p_1, \ldots, p_k$ be the first $k$ prime numbers and let $P_k:=\prod_{i=1}^k p_i$ be their product.
We choose an arbitrary residue class $r_i$ for each prime $p_i$ and consider $k$ segments of arithmetic progressions restricted to the interval $[1,P_k]$, i.e., for $p_i$ and $r_i$ we consider
$$s_i = \{n : n= h\cdot p_i + r_i, h \in \NN_0 \} \cap [1,P_k]$$

For a given choice $r_1, \ldots, r_k$ of residue classes, let
$$\gamma(n)=\# \{p_i: n\equiv r_i \pmod{p_i} \}$$
for $1 \leq n \leq P_k$ be the number of residue classes in which $n$ is contained.
We introduce some notation:
\begin{itemize}
\item An integer $n$ is \emph{covered by a prime $p_i$} if it is contained in the residue class $r_i$.
\item An integer $n$ is \emph{free} if it is not contained in any residue class, i.e., $\gamma(n)=0$.
\item An integer $n$ is \emph{available} if it is contained in at most one residue class, i.e., $\gamma(n)\leq 1$.
\item An integer $n$ is \emph{occupied} if it is contained in at least two residue class, i.e., $\gamma(n)\geq 2$.
\end{itemize}
By definition, a free integer is also available. 
Moreover, we define a second sequence $(f_k)_{k\geq 1}$ as follows:
$$ f_k = (-1)^{k} \det
 \begin{bmatrix}
 1 & 1 & \ldots & 1 & 1\\
2 & 1 & \ldots  & 1& 1 \\
1 & 3 & \ldots & 1 & 1\\
\vdots &  \vdots & \ddots & \vdots  & \vdots\\
1 & 1 & \ldots & p_{k} & 1 
\end{bmatrix}.  $$
The first five elements of this sequence for $k=1, \ldots, 5$ are
$$1, 2, 8, 48, 480 \ldots,$$
We have the following main result:

\begin{theorem} \label{thm:main}
Let $p_1, \ldots, p_k$ be the first $k$ prime numbers and let $P_k:=\prod_{i=1}^k p_i$ be their product. 
For a given choice of residue classes, i.e., one residue class $r_i$ for each prime $p_i$, the number of available elements in $[1,P_k]$ is $a_k$ and the number of free elements is $f_k$.
In other words, for every choice of residue classes, there are $a_k$ integers $n$ with $1\leq n \leq P_k$ such that $\gamma(n)\leq 1$ and $f_k$ integers with $\gamma(n)=0$.
\end{theorem}

Moreover, our method can be generalised to counting free and available integers in arbitrary sets of $k$ arithmetic progressions generated by $k$ different primes as shown in Section \ref{sec5}.

\begin{remark}
The problem of determining the number of free integers in the interval $[1,P_k]$ with residue classes $r_i=0$ for all $i$ is already mentioned by Dickson \cite[p 439]{dickson} in the first volume of his History of the Theory of Numbers. He refers to results by A. de Polignac, J. Deschamps and H.J.S Smith for which Martin \cite{martin} provides proofs in his manuscript. 
Determining the number of available integers and the connection of the two counting problems to the determinants $a_k$ and $f_k$ are new to the best of our knowledge.
\end{remark}

\begin{remark}
Lemma \ref{lem1} below gives a new characterisation of sequence $A005867$ from the The On-Line Encyclopedia of Integer Sequences \cite{oeis2} in terms of determinants.
\end{remark}

\begin{remark}
As a side product, the proof of Theorem \ref{thm:main} provides a quick way of calculating determinants of the form $a_k$ and $f_k$.
\end{remark}

The paper is organized as follows. In Section \ref{sec2} we derive structural results about $a_k$ and $f_k$. In Section \ref{sec3} we analyse the counting problem and obtain an equation that relates the available numbers after considering $k$ progressions to the available numbers after considering $k+1$ progressions. Theorem \ref{thm:main} is proven in Section \ref{sec4} and generalised in Section \ref{sec5}.

\section{Preliminaries}
\label{sec2}

We start with a structural lemma for $f_k$.

\begin{lemma} \label{lem1}
For $k\geq 2$ we have that 
$$f_k = (p_{k}-1) f_{k-1}.$$
\end{lemma}

\begin{proof}
First assume that $k$ is even. In the following we use the Laplace expansion of the initial matrix along its last row; see e.g. \cite[Chapter 1]{shilov}. Note that the first $k-1$ terms of this expansion all have determinant 0. Since $k$ is even, the sign of the element in row $k+1$ and column $k$ in the expansion is $(-1)^{k} (-1)^{k-1} = -1$ and the sign of the element in row $k+1$ and column $k+1$ is $(-1)^{k} (-1)^{k} =1$.
Hence, we have that
\begin{align*}
f_k &= (-1)^{k}\det
 \begin{bmatrix}
 1 & 1 & \ldots & 1 & 1\\
2 & 1 & \ldots  & 1& 1 \\
1 & 3 & \ldots & 1 & 1\\
\vdots &  \vdots & \ddots & \vdots  & \vdots\\
1 & 1 & \ldots & p_{k} & 1 
\end{bmatrix} \\
&=  (-1)^{k} \left( - p_{k} \det
 \begin{bmatrix}
 1 & 1 & \ldots & 1 & 1\\
2 & 1 & \ldots  & 1& 1 \\
1 & 3 & \ldots & 1 & 1\\
\vdots &  \vdots & \ddots & \vdots  & 1\\
1 & 1 & \ldots & p_{k-1} & 1 
\end{bmatrix} + \det
\begin{bmatrix}
 1 & 1 & \ldots & 1 & 1\\
2 & 1 & \ldots  & 1& 1 \\
1 & 3 & \ldots & 1 & 1\\
\vdots &  \vdots & \ddots & \vdots  & \vdots\\
1 & 1 & \ldots & p_{k-1} & 1 
\end{bmatrix}
\right )
 \\
& = (p_{k} - 1) f_{k-1}
\end{align*}
The same argument works for odd $k$.
\end{proof}

We have a second structural lemma for $a_k$.

\begin{lemma} \label{lem2}
For $k\geq 2$ we have that 
$$a_k = f_{k-1} + (p_k-1) a_{k-1}.$$
\end{lemma}

\begin{proof}
We use the following rule for determinants:
$$\det (v_1, v_2, \ldots, v_n + w) = \det (v_1, v_2, \ldots, v_n) + \det (v_1, v_2, \ldots, w)$$
in which $w, v_1, \ldots, v_n$ are row (or column vectors). We have that
\begin{align*}
a_k &= \det
 \begin{bmatrix}
2 & 1 & \ldots  & 1 \\
1 & 3 & \ldots & 1 \\
\vdots &  \vdots & \ddots & \vdots \\
1 & 1 & \ldots & p_k 
\end{bmatrix} 
= \det 
\begin{bmatrix}
2 & 1 & \ldots  & 1 & 1\\
1 & 3 & \ldots & 1 & 1\\
\vdots &  \vdots & \ddots & \vdots & \vdots \\
1 & 1 & \ldots & p_{k-1} & 1 \\
1 & 1 & \ldots & 1 & 1
\end{bmatrix} +
\det 
\begin{bmatrix}
2 & 1 & \ldots  & 1 & 1\\
1 & 3 & \ldots & 1 & 1\\
\vdots & \vdots & \ddots & \vdots & \vdots\\
1 & 1 & \ldots & p_{k-1} & 1 \\
0 & 0 & \ldots & 0& (p_k-1) 
\end{bmatrix}
\end{align*}
Now observe that
$$
\det 
\begin{bmatrix}
2 & 1 & \ldots  & 1 & 1\\
1 & 3 & \ldots & 1 & 1\\
\vdots &  \vdots & \ddots & \vdots & \vdots \\
1 & 1 & \ldots & p_{k-1} & 1 \\
1 & 1 & \ldots & 1 & 1
\end{bmatrix} = (-1)^{k-1} \det
\begin{bmatrix}
1 & 1 & \ldots & 1 &1\\
2 & 1 & \ldots  & 1 &1\\
1 & 3 & \ldots & 1 &1\\
\vdots &  \vdots & \ddots & \vdots & \vdots\\
1 & 1 & \ldots & p_{k-1} &1
\end{bmatrix} = f_{k-1}
$$
by the rules for determinants and 
$$
\det 
\begin{bmatrix}
2 & 1 & \ldots  & 1 & 1\\
1 & 3 & \ldots & 1 & 1\\
\vdots & \vdots & \ddots & \vdots & \vdots\\
1 & 1 & \ldots & p_{k-1} & 1 \\
0 & 0 & \ldots & 0& (p_k-1) 
\end{bmatrix} =  (p_k-1) a_{k-1}
$$
by the Laplace expansion along the last row. Note that the sign of diagonal elements is always positive in the Laplace expansion.
Consequently, we have that
$$a_k = f_{k-1} + (p_k-1) a_{k-1} $$
\end{proof}

\section{The counting problem}
\label{sec3}

We aim to understand the systematic intersections in the interval $[1,P_k]$ of $k$ arithmetic progressions generated by the first $k$ prime numbers $p_1, \ldots, p_k$ and arbitrary choice of residue classes $r_1, \ldots, r_k$ with $P_k = \prod_{i=1}^k p_i$.
In particular, we want to understand whether an integer $1\leq n \leq P_k$ is not covered by any residue class, is contained in at most one or in at least two residue classes.

Assume we have generated all $k$ arithmetic progression and intersected them with the interval $[1,P_k]$ obtaining the sets
$s_i$, $1\leq i \leq k$.
For each integer $1\leq n \leq P_k$ we can now ask in how many sets $s_i$ it is contained, i.e., compute $\gamma(n)$. 
We denote the number of all available integers with $av(k)$, the number of free integers with $free(k)$ and the number of integers that are contained in two or more sets with $occ(k)$, i.e.
\begin{align*}
av(k) &:= \# \{n : \gamma(n) \leq 1 \} \\
free(k) &:=\# \{n : \gamma(n) = 0 \}  \\
occ(k) &:=\# \{n : \gamma(n) > 1 \} 
\end{align*}
To illustrate this, we consider the cases $k=1$ and $k=2$ with $p_1=2$, $p_2=3$ as well as $P_1=2$ $P_2=6$.
It is easy to see that $occ(1)=0$, $free(1)=1$ and $av(1)=2$, because the arithmetic progression for $p_1=2$ is the first progression we consider and it contains one element in $[1,2]$.

Independent of the choice of residue classes, the arithmetic progression for $3$ will contain $2$ elements in $[1,6]$ and the progression for $2$ will contain $3$ elements.
By the Chinese Remainder Theorem (CRT) the two arithmetic progressions intersect in exactly one point in the interval $[1,6]$, again independent of the choice of $r_1$ and $r_2$.
Hence, three of the six integers in $[1,6]$ are covered by 2 and two integers are covered by 3. Since there is exactly one intersection, five of the 6 integers are available, i.e., $occ(2)=1$, $av(2)=5$.
Furthermore, two of the six elements are not covered independent of the choice of residue classes, i.e., $free(2)=2$ and it holds that
\begin{align*}
free(2) &= 2 = (3-1) \cdot 1 = (p_2 - 1) \cdot free(1) \\
occ(2) & = 1 = 2 + (3-1) \cdot 0 - 1 = P_{1} + (p_2-1) \cdot occ(1) - free(1)
\end{align*}
We have the following general equations. 

\begin{lemma} \label{lem3}
For $k \geq 2$ we have that
$$occ(k) =P_{k-1} + (p_k-1) \cdot occ(k-1) - free(k-1)$$
and 
$$free(k) = (p_{k}-1) \cdot free(k-1)$$
\end{lemma}

\begin{proof}
We prove the assertion by induction on $k$. It is true for $k=2$ by the above example.
Now assume it is true for $k-1$. This means after considering the first $k-1$ arithmetic progressions restricted to the interval $[1,P_{k-1}]$, we have $free(k-1)$ free integers and $occ(k-1)$ integers contained in two or more sets. Then, there are $av(k-1)=P_{k-1} - occ(k-1)$ many available integers.

Extending the $k-1$ progressions from the interval $[1,P_{k-1}]$ to $[1,P_k]$ basically generates $p_k$ copies of the pattern observed in $[1,P_{k-1}]$. Hence, we find $p_k \cdot free(k-1)$ free integers, $p_k \cdot av(k-1)$ available integers and $p_k \cdot occ(k-1)$ obstructed integers in $[1,P_k]$ before considering the $k$-th progression.

Now take an arbitrary number $x \in [1,P_{k-1}]$ and consider its $p_k$ translates
$$x,\ x+ P_{k-1}, \ldots, \ x+ (p_k-1) P_{k-1}$$
in $[1,P_k]$. The set of translates forms a permutation in $\FF_{p_k}$. In other words, independent of the residue class we choose for $p_k$, exactly one of the $p_k$ elements generated by $x$ will be contained in the $k$-th progression.

The number $x$ can be either free, available or occupied and, hence, for every free, available or occupied number from $[1,P_{k-1}]$ we get exactly one free, available or occupied number in the $k$-th progression.
In particular, this means that after adding the $k$-th progression, we find the already existing $p_k \cdot occ(k-1)$ occupied integers in $[1,P_k]$ plus the newly occupied integers, generated by those integers that were available but not free in $[1,P_{k-1}]$. Therefore, we have that
$$occ(k) = p_k \cdot occ(k-1) + av(k-1) - free(k-1).$$
In addition, before adding the $k$-th progression, we had $p_k \cdot free(k-1)$ free integers in $[1,P_k]$. However, every free integer in $[1,P_{k-1}]$ generates one integer covered by the $k$-th progression. Hence, after adding the $k$-th progression we have that
$$free(k) = (p_k-1) \cdot free(k-1).$$
\end{proof}

Using the relation $av(k)=P_{k} - occ(k)$ we can rewrite the equation as follows:
\begin{align*}
av(k) &= P_{k} - occ(k) \\
&= P_{k} - (P_{k-1} + (p_k-1) \cdot occ(k-1) - free(k-1)) \\
& = P_{k} - (P_{k-1} + (p_k-1) \cdot (P_{k-1} - av(k-1))  - free(k-1))\\
& = (p_k-1) \cdot av(k-1) + free(k-1)
\end{align*}

\section{Proof of Theorem \ref{thm:main}}
\label{sec4}

We have already seen that $a_1=2 = av(1)$ and $f_1 = 1 = free(1)$. 
Furthermore, for $k=2$ we have that
$$
a_2 = \det
 \begin{bmatrix}
2 & 1  \\
1 & 3  
\end{bmatrix} = 2 \cdot 3 - 1 = 5
$$
is the number of available integers, which we have determined in the example in Section \ref{sec3}
and
$$
f_2 = (-1)^{2} \det
 \begin{bmatrix}
1 & 1 & 1 \\
2 & 1  & 1\\
1 & 3  & 1
\end{bmatrix} =1 + 1 + 2 \cdot 3 - 2 - 3 -1 = 2
$$
gives exactly the number of free elements, i.e., those integers in $[1,6]$ that are not covered by any of the two primes.
To make things more interesting we consider one more example. We have that
$$
a_3 = \det
 \begin{bmatrix}
2 & 1  & 1\\
1 & 3  & 1 \\
1 & 1 & 5
\end{bmatrix} = 30 + 1 + 1 - 2 - 3 - 5 = 22
$$
and we see that
$$ a_3 = 22 = 2 + (5-1) \cdot 5 = f_2 + (p_3 - 1) \cdot a_2 $$

To turn to the general step, assume that $a_k$ is the number of available elements in $[1,P_k]$ and $f_k$ is the number of free elements. By Lemma \ref{lem3} we know that 
$$av(k+1) = (p_k-1) \cdot av(k) + free(k).$$
Furthermore, by Lemma \ref{lem2} we know that 
$$a_{k+1} = (p_k-1) a_{k} + f_{k} $$
and, therefore, we conclude that $av(k+1) = a_{k+1}$.

Similarly, it was shown in Lemma \ref{lem3} that
$$free(k+1) = (p_k-1) \cdot free(k)$$
and by Lemma \ref{lem1} we have
$$f_{k+1} = (p_{k}-1) f_{k}.$$
Hence, we conclude that $free(k+1)=f_{k+1}$ and the theorem is proven.

\section{Generalisation}
\label{sec5}

Note that our argument is purely structural and we did not use the explicit values of the first $k$ primes. We only used explicit values for the base case of the theorem. However, this base case can easily be generalised as the next lemmas shows. Moreover, we used the prime property when we considered permutations of $\FF_{p_k}$. But again, no explicit value of $p_k$ was used. Hence, it is possible to extend our method to count free and available integers in any set of $k$ arithmetic progressions generated by $k$ different primes and arbitrary residue classes.

We define
$$ A_k(p_{j_1}, \ldots, p_{j_k}) = \det
 \begin{bmatrix}
p_{j_1} & 1 & \ldots  & 1 \\
1 & p_{j_2} & \ldots & 1 \\
\vdots &  \vdots & \ddots & \vdots \\
1 & 1 & \ldots & p_{j_k}
\end{bmatrix}
$$
and 
$$  
F_k(p_{j_1}, \ldots, p_{j_k}) =(-1)^k \det
 \begin{bmatrix}
 1 & 1 & \ldots & 1 & 1\\
p_{j_1} & 1 & \ldots  & 1& 1 \\
1 & p_{j_2} & \ldots & 1 & 1\\
\vdots &  \vdots & \ddots & \vdots  & \vdots\\
1 & 1 & \ldots & p_{j_k} & 1 
\end{bmatrix}. 
$$

Then we have
\begin{lemma}\label{lem5}
Let $p_i, p_j$ be two primes and $r_i, r_j$ an arbitrary choice of residue classes for these primes. The number of available integers in $[1,p_i p_j]$ after the two progressions have been restricted to the interval is $A_2(p_i, p_j)$ and the number of free integers is $F_2(p_i, p_j)$.
\end{lemma}

\begin{proof}
We have that $A_2(p_i, p_j) = p_i p_j - 1$ and the assertion follows by the same argument as in the example for $p_i=2$ and $p_j=3$. Similarly, $F_2(p_i, p_j) = p_i \cdot p_j - p_i - p_j +1 $ from which the assertion follows by the principle of inclusion-exclusion.
\end{proof}

\begin{lemma}\label{lem5}
Let $p_i, p_j, p_k$ be three primes and $r_i, r_j, r_k$ an arbitrary choice of residue classes for these primes. The number of available integers in $[1,p_i p_j p_k]$ after the $3$ progressions have been restricted to the interval is $A_3(p_i, p_j, p_k)$.
\end{lemma}

\begin{proof}
We have that
$$
A_3(p_i, p_j, p_k) = \det
 \begin{bmatrix}
p_i & 1 &1 \\
1 & p_j  &1 \\
1 & 1 & p_k
\end{bmatrix} = p_i \cdot p_j \cdot p_k +1 + 1 - p_i - p_j - p_k
$$
We interpret this as follows for $P=p_i p_j p_k$.
The interval $[1,P]$ can be partitioned in the following ways:
\begin{align*}
[1,P] &= [1,p_i p_j] \cup [p_i p_j +1, 2p_i p_j] \cup \ldots \cup [(p_k -1) p_i p_j,P] \\
&=[1,p_i p_k] \cup [p_i p_k +1, 2p_i p_k] \cup \ldots \cup [(p_j -1) p_i p_k,P] \\
& = [1,p_k p_j] \cup [p_k p_j +1, 2p_k p_j] \cup \ldots \cup [(p_i -1) p_k p_j,P]
\end{align*}
The first partition corresponds to the pair of primes $(p_i,p_j)$ and by the CRT there is exactly one integer covered by both primes in each interval of the partition. Since there are $p_k$ such intervals, there are $p_k$ integers with $\gamma(n)>1$.
The next partition corresponds to the pair $(p_i,p_k)$ contributing $p_j$ integers with $\gamma(n)>1$ and the third to the pair $(p_k,p_j)$ contributing $p_i$ such integers.
Finally, by the CRT we know that there is exactly one point with $\gamma(n)=3$ in $[1,P]$, which is counted in each of the above cases. Hence, by the principle of inclusion exclusion we get that there are
$$p_i p_j p_k - p_i - p_j - p_k +1 +1$$
available integers in the interval $[1,p_i p_j p_k]$.
\end{proof}

Hence, we have that
\begin{align*}
A_3(p_i, p_j, p_k) &= p_i  p_j  p_k +1 + 1 - p_i - p_j - p_k \\
&= p_i  p_j p_k + p_i p_j - p_i p_j +1 + 1 - p_i - p_j - p_k \\
&= p_i  p_j - p_i - p_j +1 + (p_k - 1) (p_i p_j - 1) \\
&= F_2(p_i, p_j) + (p_k-1) A_2(p_i, p_j)
\end{align*}

Furthermore, Lemmas \ref{lem1}, \ref{lem2} and \ref{lem3} can be generalised to imply the following theorem.

\begin{theorem} \label{thm2}
Let $p_{j_1}, \ldots, p_{j_k}$ be $k$ distinct prime numbers and let $P_k:=\prod_{i=1}^k p_{j_i}$ be their product. 
For $k\geq 2$ and a given choice of residue classes, i.e., one residue class $r_i$ for each prime $p_i$, the number of available elements in $[1,P_k]$ is $A_k(p_{j_1}, \ldots, p_{j_k})$ and the number of free elements is $F_k(p_{j_1}, \ldots, p_{j_k})$.
In other words, for every choice of residue classes, there are $A_k(p_{j_1}, \ldots, p_{j_k})$ integers $n$ with $1\leq n \leq P_k$ such that $\gamma(n)\leq 1$ and $F_k(p_{j_1}, \ldots, p_{j_k})$ integers with $\gamma(n)=0$.
\end{theorem}



\begin{thebibliography}{1}

\bibitem{dickson} L. E. Dickson, History of the Theory of Numbers. Vol. I: Divisibility and primality. Reprint of the 1919 original published by Carnegie Institution, Washington, DC. Mineola, NY: Dover Publications (2005).

\bibitem{martin} D. R. Martin, Proofs Regarding Primorial Patterns, available online: https://oeis.org/A005867/a005867.pdf (15/7/2025)

\bibitem{oeis} The Online Encyclopedia of Integer Sequences, https://oeis.org/A067549 (15/7/2025)

\bibitem{oeis2} The Online Encyclopedia of Integer Sequences, https://oeis.org/A005867 (15/7/2025)

\bibitem{shilov} G. Shilov, Linear Algebra, Translated and edited by R.A. Silverman. Revised English edition, Englewood Cliffs, N.J., Prentice-Hall, (1971).

\end{thebibliography}
\end{document}